\documentclass{article}
\usepackage{algorithm,algorithmic,multirow,amsmath,subfigure}
\usepackage{mathrsfs,extarrows,cases}
\usepackage{enumerate,amssymb,amsthm}
\usepackage{float}

\newtheorem{Definition}{Definition}
\newtheorem{Theorem}{Theorem}
\newtheorem{Proposition}{Proposition}
\newtheorem{Lemma}{Lemma}

\usepackage{color,xcolor,tikz}
\begin{document}

\title{$p$-order Tensor Products with Invertible Linear Transforms}
\author{Jun Han \thanks{School of Mathematical Sciences, Nankai University, Tianjin, China. 1610021@mail.nankai.edu.cn.}}
\date{\today}

\maketitle
\begin{abstract}
  This paper studies the issues about tensors. Three typical kinds of tensor decomposition are mentioned. Among these decompositions, the t-SVD is proposed in this decade. Different definitions of rank derive from tensor decompositions. Based on the research about higher order tensor t-product and tensor products with invertible transform, this paper introduces a product performing higher order tensor products with invertible transform, which is the most generalized case so far. Also, a few properties are proven. Because the optimization model of low-rank recovery often uses the nuclear norm, the paper tries to generalize the nuclear norm and proves its relation to multi-rank of tensors. The theorem paves the way for low-rank recovery of higher order tensors in the future.
\end{abstract}

\section{Introduction}

Nowadays, tensor-valued high dimensional data can be often observed in many areas, such as seismic data \cite{ely20155d},hyperspectral image \cite{chang2017weighted}, and video de-noising \cite{zhang2014novel}. Matrix-based methods may damage the inherent structure of data and hide some connection inside data by vectorizing or matricizing. So it is necessary for us to do research about tensors. There are three typical kinds of tensor decompositions, i.e., CP \cite{hitchcock1927expression,kiers2000towards}, Tucker decomposition \cite{tucker1966some,kruskal1989rank} and t-SVD \cite{kilmer2011factorization}. We refer the readers to \cite{kolda2009tensor} for a thorough review of CP and Tucker decomposition. The rank derived by CP decomposition is NP hard to compute \cite{haastad1989tensor}. Tucker decomposition may not give the best approximation of tensors. In this decade, there is a new tensor product called t-product \cite{kilmer2011factorization}. Most of conclusions in the matrix case can be generalized to tensors using this product. We can derive a generalized matrix SVD in the case of tensor based on t-product. Authors in \cite{kilmer2013third} analyzed this product in an operator perspective and proposed more concepts in a generalized version. A special case of $p$-order tensors ($p>3$) product is given \cite{martin2013order}.They mainly use the concept of matrix slices. The generalized t-product is given for the third order tensor based on any invertible transform \cite{kernfeld2015tensor}. Now we see many applications of this product \cite{ely20155d,chang2017weighted,jiang2019robust}. This is the most natural kind of product corresponding to matrix cases. In this paper, I generalized this product to the case of $p$-order tensors with invertible linear transforms. To the best of my knowledge, this is the most generalized case explicitly presented so far. Moreover, different definitions of tensor ranks, derived from tensor decompositions, facilitate the research about low-rank recovery. The nature of low-rank means that data such as images,videos, texts, all lie on low-dimensional subspaces \cite{belkin2003laplacian,eckart1936approximation}. In the matrix case, there are many efficient methods about low-rank recovery like \cite{li2013compressed}. Among these methods, convex programming with the matrix nuclear norm is very popular.
I summarize the main contribution of the paper as follows,

$\bullet$ I propose a $p$-order tensor product using linear invertible transforms.I extend and generalize the work \cite{kernfeld2015tensor,martin2013order}. I also give some related definitions and prove some properties.

$\bullet$ I give a definition of $p$-order tensor nuclear norm and prove a theorem to show the reasonableness of this norm, which is a generalization of \cite{song2019robust}.We can use it to do research about $p$-order tensor low-rank recovery in the future.

\subsection{Notation}
In this paper, the vector is denoted like $\mathbf{a}$ and its $i$-th element is like $a_i$. A matrix, which is also a second order tensor, is like $\mathbf{A}$, whose element is like $\mathbf{A}_{i,j}$. I denote the symbol like $\mathcal{A}$ as three or higher order tensors. There are three kinds of slices in third order tensors: horizontal slice $\mathcal{A}(i,:,:)$,lateral slice$\mathcal{A}(:,i,:)$, and frontal slice $\mathcal{A}(:,:,i)$. I will also denote the frontal slice as $\mathcal{A}^{(i)}$ for simplicity. By fixing two indices of third order tensors, we get the fiber. The mode-3 fiber is also called tube, denoted as $\mathcal{A}(i,j,:)$. In this paper, I will denote $\mathring{\mathbf{a}}$ as a tube of tensor $\mathcal{A}$. We may vectorize a tube by $\mathbf{a}=\mathrm{vec}(\mathring{\mathbf{a}})$. The lateral slice of a tensor is like $\vec{\boldsymbol{b}}$. For $p$ order tensors, the $(i_1,i_2,\cdots,i_p)$-th element of $\mathcal{A}$ is denoted as $\mathcal{A}_{i_1,i_2,\cdots,i_p}$. Here I would like to review the matrix and tensor norms. The matrix spectral norm of $||\mathbf{A}||$ is $\max_i \sigma_i(\mathbf{A})$ and nuclear norm $||\mathbf{A}||_*$ is $\sum_i \sigma_i(\mathbf{A})$. $\ell_1$ norm of a tensor $||\mathcal{A}||_1$ is $\sum_{ijk}|\mathcal{A}_{i,j,k}|$ and $\ell_{\infty}$ norm is denoted as $||\mathcal{A}||_{\infty}=\max_{ijk}|\mathcal{A}_{i,j,k}|$.
I may use the $m$-mode matrix product to define the generalized tensor product in this paper.
\begin{Definition}
  \cite{kolda2009tensor}
  $\mathcal{A}\in \mathbb{C}^{n_1 \times n_2 \times \cdots\times n_p}$,$\mathbf{X}\in \mathbb{C}^{J \times n_m}$
  \[
  (\mathcal{A}\times_m \mathbf{X})_{i_1,\cdots,i_{n-1},j,i_{n+1},\cdots,i_p}=\sum_{i_m=1}^{n_m}
  \mathcal{A}_{i_1,i_2,\cdots,i_p}\mathbf{X}_{j,i_m},
  \]
\end{Definition}

\subsection{Review of t-product}
The t-product is mainly based on discrete Fourier transform(DFT)\cite{kilmer2011factorization}.
We denote $\hat{\mathcal{A}}$ as the DFT of every tube of $\mathring{\mathbf{a}}$.
Using the MATLAB operation, we can get $\hat{\mathcal{A}}$ by $\mathtt{fft}(\mathcal{A},[ ],3)$.
Specifically, we exert the DFT on $\mathrm{vec}(\mathring{\mathbf{a}})$.
The formal definition of t-product for third order tensors is as follows,
\begin{Definition}
  \cite{kilmer2011factorization}For $\mathcal{A}\in \mathbb{R}^{n_1 \times n_2\times n_3},\mathcal{B}\in \mathbb{R}^{n_2 \times n_4\times n_3}$,
  \[
  \mathcal{A} *\mathcal{B}=\mathtt{fold}(\mathtt{bcirc}(\mathcal{A})  \mathtt{MatVec}(\mathcal{B})).
  \]
  We get a tensor in $\mathbb{R}^{n_1 \times n_4\times n_3}$.
  \end{Definition}
  Here note that
  \[
  \mathtt{bcirc}(\mathcal{A})=
  \left(
  \begin{array}{ccccc}
    \mathcal{A}^{(1)} & \mathcal{A}^{(n_3)} & \mathcal{A}^{(n_3-1)} & \cdots & \mathcal{A}^{(2)} \\
    \mathcal{A}^{(2)} & \mathcal{A}^{(1)} & \mathcal{A}^{(n_3)} & \cdots & \mathcal{A}^{(3)}\\
    \vdots & \ddots & \ddots & \ddots & \vdots \\
    \mathcal{A}^{(n_3)} & \mathcal{A}^{(n_3-1)} & \cdots & \mathcal{A}^{(2)} & \mathcal{A}^{(1)}
  \end{array}
  \right),
  \mathtt{MatVec}(\mathcal{A})=
  \left(
  \begin{array}{c}
    \mathcal{A}^{(1)} \\
    \mathcal{A}^{(2)}\\
    \vdots \\
    \mathcal{A}^{(n_3)}
  \end{array}
  \right),
  \]
 and $\mathtt{fold}$ is the inverse operation of $\mathtt{MatVec}$.
$\mathtt{bcirc}(\mathtt{MatVec}(\mathcal{A}))$ can be diagonalized in block,
   \begin{equation}\label{bdiagdft}
(\mathbf{F}_{n_3}\otimes\mathbf{I}_{n_1})\mathtt{circ}(\mathtt{MatVec}
(\mathcal{A}))(\mathbf{F}_{n_3}^H\otimes\mathbf{I}_{n_2})=
\mathtt{bdiag}(\hat{\mathcal{A}}),
\end{equation}
where $\otimes$ is the Kronecker product.
The block diagonal elements of $\mathtt{bdiag}(\hat{\mathcal{A}})$ are the frontal slices of $\hat{\mathcal{A}}$.
We can find that the t-product can be computed more efficiently than CP and Tucker decomposition.
Definitions of transpose of tensors and orthogonal tensors are given. Note that we need to reverse the order of the frontal slices except the first when transposing.
\begin{Definition}
  \cite{kilmer2011factorization}
  The transpose of $\mathcal{A}\in \mathbb{R}^{n_1 \times n_2\times n_3}$ is $\mathcal{A}^\top\in \mathbb{R}^{n_2 \times n_1\times n_3}$, which satisfies
  \[
  (\mathcal{A}^\top)^{(1)}=(\mathcal{A}^{(1)})^\top,(\mathcal{A}^\top)^{(i)}=
  (\mathcal{A}^{(n_3+2-i)})^\top,i=2,3,\cdots,n_3.
  \]
\end{Definition}
\begin{Definition}
  \cite{kilmer2011factorization}$\mathcal{Q}\in \mathbb{R}^{n \times n\times n_3}$ is the orthogonal tensor if
  $\mathcal{Q}^\top *\mathcal{Q}=\mathcal{Q} *\mathcal{Q}^\top=\mathcal{I}$, where $\mathcal{I}$ is the identity tensor,
  \[
  \mathcal{I}^{(1)}=\mathbf{I}_n,\mathcal{I}^{(i)}=\mathbf{O},i=2,3,\cdots,n_3,
  \]
  $\mathbf{I}_n$ is the identity matrix. All elements of $\mathbf{O}$ are zero.
\end{Definition}
From the definition above, we can also know the definition of identity tensors. Diagonal tensors should be different from \cite{kolda2009tensor} to derive the theorem of t-SVD.
\begin{Definition}
  \cite{kilmer2011factorization}$\mathcal{A}\in \mathbb{R}^{n_1 \times n_2\times n_3}$ is f-diagonal if all frontal slices of $\mathcal{A}^{(i)}$ is diagonal.
\end{Definition}
\begin{Theorem}
  \cite{kilmer2011factorization}
  Given $\mathcal{A}\in \mathbb{R}^{n_1 \times n_2\times n_3}$, there exist orthogonal tensors
  $\mathcal{U}\in \mathbb{R}^{n_1 \times n_1\times n_3},\mathcal{V}\in \mathbb{R}^{n_2 \times n_2\times n_3}$,and a f-diagonal tensor $\mathcal{S}\in \mathbb{R}^{n_1 \times n_2\times n_3}$ subject to $\mathcal{A}=\mathcal{U}*\mathcal{S}*\mathcal{V}^\top$.
\end{Theorem}
So the t-SVD is the generalization of the matrix SVD.
\begin{Definition}
  \cite{kilmer2013third}
  The multi-rank of $\mathcal{A}\in \mathbb{C}^{n_1\times n_2\times n_3}$ is a vector given by
  \[
  \mathbf{rank}_m(\mathcal{A})=
  (\mathrm{rank}(\hat{\mathcal{A}}^{(1)}),\mathrm{rank}
  (\hat{\mathcal{A}}^{(2)}),
  \cdots,\mathrm{rank}(\hat{\mathcal{A}}^{(n_3)}))^\top.
  \]
\end{Definition}
$\ell_1$ and $\ell_2$ norm of multi-rank may describe the complexity or sparsity of tensors. I will give its generalization in Section 3.
\begin{Definition}
  \cite{zhang2014novel,lu2019tensor}
  The tubal-rank of $\mathcal{A}\in \mathbb{R}^{n_1\times n_2\times n_3}$ is the number of non-zero tubes of $\mathcal{S}$ where $\mathcal{A}=\mathcal{U}*\mathcal{S}*\mathcal{V}^\top$. We can denote as
  \[
  \mathrm{rank}_t(\mathcal{A})=\#\{i:\mathrm{vec}(\mathcal{S}(i,i,:))\neq \vec{0}\}=\#\{i:\mathcal{S}(i,i,1)\neq 0\}.
  \]
\end{Definition}
To define the t-product under invertible linear transform, we must give the triangle operator. Consider $\mathcal{A}\in\mathbb{C}^{m\times l\times n},\mathcal{B}\in\mathbb{C}^{l\times q\times n}$,
\begin{Definition}
  \cite{kernfeld2015tensor}The i-th frontal slice of $\mathcal{A}\triangle\mathcal{B}$ is $\mathcal{A}^{(i)}\mathcal{B}^{(i)}$.
\end{Definition}
\begin{Definition}
  \cite{kernfeld2015tensor}
  If $L:\mathbb{C}^{1\times l\times n}\rightarrow\mathbb{C}^{1\times l\times n}$ is an invertible linear transform,
  define $\mathrm{vec}(L(\mathring{\mathbf{a}}))=\mathbf{T}\mathbf{a}$, where $\mathbf{T}$ is the corresponding matrix of transform $L$.
\end{Definition}
We need to perform the above operation in every tube of tensors. Here we give the definition of tensor products with invertible linear transforms.
\begin{Definition}
  \cite{kernfeld2015tensor}$*_L$ product is as follows, $L(\mathcal{A}*_L\mathcal{B})=L(\mathcal{A})\triangle L(\mathcal{B})$.
\end{Definition}
One special case of $*_L$ is t-product.
\begin{Definition}
  \cite{zhang2016exact,semerci2014tensor}The tensor spectral norm and tensor nulcear norm are defined as
  \[
  \|\mathcal{A}\|=\|\mathtt{bdiag}(\hat{\mathcal{A}})\|,\quad
  \|
   \mathcal{A}\|_*=\sum_{i=1}^{n_3}\|\hat{\mathcal{A}}^{(i)}\|.
  \]
\end{Definition}
The convex envelope is often used in the low-rank recovery model because of non-convexity of tensor ranks. For example, one theorem about the convex envelope of the tensor rank is given in \cite{lu2019tensor}.
Because in the next section I will give the generalized product for $p$-order tensors, I will not give too many details of the previous work.

\section{$p$-order Tensors $*_L$}
\begin{figure}
  \centering
  \includegraphics[width=12cm]{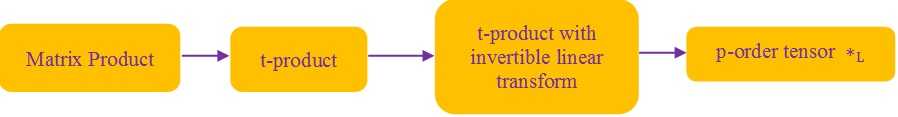}
  \caption{Relationship between Different Products}\label{product}
\end{figure}
I use the idea of \cite{martin2013order} and also introduce the matrix slices of $p$-order tensors. We can see the Figure \ref{product} showing the relationship between different products. The product on the right is the generalization of that on the left.
I denote the matrix slice of $\mathcal{A}$ as $\mathcal{A}^{'}$, a third order tensor.
We will discuss $p$-order tensors $\mathcal{A}\in \mathbb{C}^{n_1\times l\times \cdots\times n_p}$ and $\mathcal{B}\in \mathbb{C}^{l\times n_2\times \cdots\times n_p}$.
First define $\triangle_p$ as a generalized case of $\triangle$,
\begin{Definition}
  \[
  (\mathcal{A}\triangle_p \mathcal{B})^{'(i)}=\mathcal{A^{'}}^{(i)}
  \mathcal{B^{'}}^{(i)},i=1,2,\cdots,n_3 n_4\cdots n_p.
  \]
\end{Definition}
Now I can introduce the $p$-order tensor product with invertible linear transform,
\begin{Definition}
Given invertible transform $L$, i.e. $L(\mathcal{A})=A\times_3 \mathbf{T}_{n_3}\times_4 \mathbf{T}_{n_4}\cdots\times_{p} \mathbf{T}_{n_p}$,define the product as $L(\mathcal{A}*_L\mathcal{B})=L(\mathcal{A})\triangle_p
L(\mathcal{B})$.
In particular, if $L$ is a unitary transform, we denote the product as $*_u$.
$\mathbf{T} _{n_i},i=3,4,\cdots,n_p$ is the corresponding matrix of transform $L$.
\end{Definition}
As \cite{martin2013order} shows that
\begin{equation}\label{bdiagdftnd}
(\mathbf{F}_{n_p}\otimes\mathbf{F}_{n_{p-1}}\otimes
\cdots\otimes\mathbf{F}_{n_3}\otimes \mathbf{I}_{n_1})
\cdot\tilde{\mathbf{A}}\cdot
(\mathbf{F}_{n_p}^{(-1)}\otimes\mathbf{F}_{n_{p-1}}^{(-1)}\otimes
\cdots\otimes\mathbf{F}_{n_3}^{(-1)}\otimes \mathbf{I}_{n_2}),
\end{equation}
we can diagonalize $\tilde{\mathbf{A}}$ in blocks. We may perform many times of $\mathtt{MatVec}$ to obtain $\tilde{\mathbf{A}}$ from $p$-order tensors. The diagonal block is like $(L(\mathcal{A}))^{'(i)}$. Note that $L^{-1}(\mathcal{A})=A\times_{p} \mathbf{T}_{n_p}^{-1}\times_{n_{p-1}} \mathbf{T}_{n_{p-1}}^{-1}\cdots\times_{3} \mathbf{T}_{3}^{-1}$.
I also give an algorithm to compute the product of $p$-order tensors,
\begin{algorithm}[H]
\caption{Calculate $*_L$ product of $p$-order tensors}
\label{alg1}
\begin{algorithmic}[1]
\REQUIRE $\mathcal{A}\in \mathbb{R}^{n_1\times l\times \cdots\times n_p}$,$\mathcal{B}\in \mathbb{R}^{l\times n_2\times \cdots\times n_p}$
\ENSURE $\mathcal{C}\in \mathbb{R}^{n_1\times n_2\times \cdots\times n_p}$
\FOR{$i=3,\cdots,p$}
\STATE$\mathcal{A}=\mathcal{A}\times_i \mathbf{T}_{n_i};\mathcal{B}=\mathcal{B}\times_i \mathbf{T}_{n_i};$
\ENDFOR
\FOR{$i=1,2,\cdots,n_3 n_4 \cdots n_p$}
\STATE$\mathcal{C}^{'(i)}=\mathcal{A}^{'(i)}\mathcal{B}^{'(i)};$
\ENDFOR
\FOR{$i=p,p-1,\cdots,3$}
\STATE$\mathcal{C}=\mathcal{C}\times_i \mathbf{T}_{n_i}^{-1}$
\ENDFOR
\end{algorithmic}
\end{algorithm}
I give a generalized definition of identity tensors,invertible tensors,
\begin{Definition}
  $\mathcal{I}\in \mathbb{C}^{n\times n\times n_3 \cdots\times n_p}$ satisfying
  $L(\mathcal{I})^{'(i)}=\mathbf{I}_n,i=1,2,\cdots,n_3\cdots n_p$ is an identity tensor.
\end{Definition}
It is easy to verify that
\begin{equation}\label{identityt}
L(\mathcal{I})\triangle_p L(\mathcal{A})=
L(\mathcal{A})\triangle_p L(\mathcal{I})=L(\mathcal{A}).
\end{equation}
Then we get $
\mathcal{A}*_L\mathcal{I}=\mathcal{I}*_L\mathcal{A}=\mathcal{A}$.
\begin{Definition}
  $\mathcal{A}\in \mathbb{C}^{n\times n\times n_3\times \cdots\times n_p}$ is invertible if there exists $\mathcal{A}^{-1}\in\mathbb{C}^{n\times n\times n_3\times \cdots\times n_p}$ subject to
  \[
  \mathcal{A}*_L\mathcal{A} ^{-1}=\mathcal{A}^{-1}*_L\mathcal{A}=\mathcal{I}.
  \]
\end{Definition}

Some theoretical results can be shown,
\begin{Lemma}\label{jiehe}
  The $p$-order tensors is associative under $*_L$.
\end{Lemma}
\begin{proof}
  \begin{equation}\label{ptensorjiehe}
\begin{split}
   (\mathcal{A}*_L\mathcal{B})*_L\mathcal{C} &= L^{-1}(L(\mathcal{A})\triangle_p L(\mathcal{B}))*_L \mathcal{C}\\
     &=L^{-1}(L(\mathcal{A})\triangle_p L(\mathcal{B}))*_L \mathcal{C}\\
     &=L^{-1}(L(L^{-1}(L(\mathcal{A})\triangle_p L(\mathcal{B})))\triangle_p L(\mathcal{C})) \\
     &=L^{-1}((L(\mathcal{A})\triangle_p L(\mathcal{B}))\triangle_p L(\mathcal{C})) \\
     &=L^{-1}(L(\mathcal{A})\triangle_p L(L^{-1}(L(\mathcal{B})\triangle_p L(\mathcal{C})))) \\
     &=L^{-1}(L(\mathcal{A})\triangle_p (L(\mathcal{B})*_L \mathcal{C}))\\
     &=\mathcal{A}*_L (\mathcal{B}*_L \mathcal{C}).
\end{split}
\end{equation}
\end{proof}
Tubes of $p$-order tensors are in $\mathbb{C}^{1\times 1\times n_3\times \cdots\times n_p}$,also denoted as $\mathring{\mathbf{a}}$.
\begin{Theorem}
  The set of all tubes in $\mathbb{C}^{1\times 1\times n_3\times \cdots\times n_p}$ form a commutative ring under $*_L$.
\end{Theorem}
\begin{proof}
  From the definition of $*_L$, we can show
\begin{equation}\label{jiaohuan}
\begin{split}
   \mathring{\mathbf{a}}*_L(\mathring{\mathbf{b}}+\mathring{\mathbf{c}}) &=L^{-1}(L(\mathring{\mathbf{a}})\triangle_p L(\mathring{\mathbf{b}}+\mathring{\mathbf{c}}))\\
     &=L^{-1}(L(\mathring{\mathbf{a}})\triangle_p L(\mathring{\mathbf{b}})+L(\mathring{\mathbf{a}})\triangle_p
  L(\mathring{\mathbf{c}})) \\
     &=L^{-1}(L(\mathring{\mathbf{a}})\triangle_p L(\mathring{\mathbf{b}}))+L^{-1}(L(\mathring{\mathbf{a}})\triangle_p
  L(\mathring{\mathbf{c}})) \\
     &=\mathring{\mathbf{a}}*_L \mathring{\mathbf{b}}+\mathring{\mathbf{a}}*_L \mathring{\mathbf{c}}.
\end{split}
\end{equation}
The proof of $(\mathring{\mathbf{a}}+\mathring{\mathbf{b}})*_L\mathring{\mathbf{c}}=
\mathring{\mathbf{a}}*_L\mathring{\mathbf{c}}+\mathring{\mathbf{b}}*_L
\mathring{\mathbf{c}}$ is similar to the above.
We also know that
\begin{equation}\label{jiaohuanring}
  \mathring{\mathbf{a}}*_L\mathring{\mathbf{b}}=L^{-1}(L(\mathring{\mathbf{a}})
\triangle_p L(\mathring{\mathbf{b}}))=L^{-1}(L(\mathring{\mathbf{b}})
\triangle_p L(\mathring{\mathbf{a}}))=\mathring{\mathbf{a}}*_L\mathring{\mathbf{b}}.
\end{equation}
\end{proof}

\begin{Definition}[Transpose]
  $\mathcal{A}^H\in \mathbb{C}^{l\times n_1\times n_3\times\cdots\times n_p}$ is the transpose of $\mathcal{A}$ if $L(\mathcal{A} ^H)^{'(i)}=(L(\mathcal{A})^{'(i)})^H$.
\end{Definition}
The following propositions shows that the transpose of tensors keeps some properties of transposing matrices.
\begin{Proposition}
  $(\mathcal{A}*_L\mathcal{B})^H=\mathcal{B}^H*_L\mathcal{A}^H$
\end{Proposition}
\begin{proof}
  \begin{equation}\label{ptensorH}
  \begin{split}
     L(\mathcal{B}^H*_L\mathcal{A}^H)^{'(i)}&=(L(\mathcal{B}^H)\triangle_p L(\mathcal{A}^H))^{(i)} \\
       &=L(\mathcal{B}^H)^{'(i)}L(\mathcal{A}^H)^{'(i)} \\
       &=(L(\mathcal{B})^{'(i)})^H (L(\mathcal{A})^{'(i)})^H \\
       &=(L(\mathcal{A})^{'(i)}L(\mathcal{B})^{'(i)})^H\\
       &=((L(\mathcal{A})\triangle_p L(\mathcal{B}))^{(i)})^H\\
       &=(L(\mathcal{A}*_L \mathcal{B})^{'(i)})^H.
  \end{split}
  \end{equation}
  So we get $L(\mathcal{B}^H*_L\mathcal{A}^H)=L((\mathcal{A}*_L \mathcal{B})^H)$.
\end{proof}
The definition of a unitary tensor is given here,
\begin{Definition}
  $\mathcal{Q}\in \mathbb{C}^{m\times m\times n_3\times \cdots\times n_p}$ is a unitary tensor if $\mathcal{Q}^H *_L\mathcal{Q}=\mathcal{Q}*_L\mathcal{Q}^H=\mathcal{I}$.
\end{Definition}
The following theorem is important.
\begin{Theorem}[$p$-order tensor t-SVD under $*_L$]
  There exists unitary tensors $\mathcal{U}\in\mathbb{C}^{n_1\times n_1\times n_3\times \cdots\times n_p}$,
  $\mathcal{V}\in \mathbb{C}^{l\times l\times n_3\times \cdots\times n_p}$ and $\mathcal{S}$ whose dimension is equal to that of $\mathcal{A}$ subject to $\mathcal{A}=\mathcal{U}*_L\mathcal{S}*_L\mathcal{V}^H$.
  Also, $\mathcal{S}_{i_1 i_2\cdots i_p}\neq 0$ if and only if $i_1=i_2$.
  \label{p,tSVD}
\end{Theorem}
\begin{proof}
  Let $\hat{\mathcal{A}}=L(\mathcal{A})$. Compute the matrix SVD of
  $\hat{\mathcal{A}} ^{'(i)}$, i.e.
  \[
  \hat{\mathcal{A}} ^{'(i)}=\hat{\mathcal{U}}^{'(i)}\hat{\mathcal{S}}^{'(i)}
  (\hat{\mathcal{V}}^{'(i)})^H.
   \]
   Let $\mathcal{U}=L^{-1}(\hat{\mathcal{U}})$.
  $\mathcal{S}=L^{-1}(\hat{\mathcal{S}})$ and $\mathcal{V}=
  L^{-1}(\hat{\mathcal{V}})$. We need to prove that $\mathcal{U}$ is a unitary tensor. The proof of $\mathcal{V}$ is similar.
  \begin{equation}\label{ortho}
  \begin{split}
  L(\mathcal{U}*_L\mathcal{U} ^H)^{'(i)}&=(L(\mathcal{U})\triangle_p L(\mathcal{U}^H))^{(i)}=L(\mathcal{U})^{'(i)}L(\mathcal{U^H})^{'(i)}\\
  &=\hat{\mathcal{U}}^{'(i)}(\hat{\mathcal{U}}^{'(i)})^H
  =\mathbf{I}_{n_1}=
  L(\mathcal{I})^{'(i)},
  \end{split}
  \end{equation}
  So $\mathcal{U}*_L\mathcal{U}^H=\mathcal{I}$ and the conclusion follows.
\end{proof}

If $\mathcal{C}=\mathcal{A}^H*_L\mathcal{A}$, $L(\mathcal{C})^{'(i)}$ is positive semidefinite because
\begin{equation}\label{semidefinite}
  L(\mathcal{C})^{'(i)}=L(\mathcal{A}^H)^{'(i)}L(\mathcal{A})^{'(i)}=
(L(\mathcal{A})^{'(i)})^H L(\mathcal{A})^{'(i)}.
\end{equation}

The tensor determinant can be defined but now it is not a number.
\begin{Definition}[Tensor Determinant]
  The determinant, which is in
  $\mathbb{C}^{1\times 1\times n_3\times \cdots \times n_p}$, of $\mathcal{A}\in \mathbb{C}^{2\times 2\times n_3\times \cdots \times n_p}$ can be defined as $\mathtt{det}(\mathcal{A})=|\mathcal{A}|=\mathring{\mathbf{a}}_{11}
  *_L\mathring{\mathbf{a}}_{22}-\mathring{\mathbf{a}}_{12}
  *_L\mathring{\mathbf{a}}_{21}$ where $\mathring{\mathbf{a}} _{ij}=\mathcal{A}(i,j,:,\cdots,:)$.
  We can define the determinant of
  $\mathcal{A}\in \mathbb{C}^{3\times 3\times n_3\times \cdots \times n_p}$ by unfolding the rows,
  \[
  |\mathcal{A}|=\mathring{\mathbf{a}}_{11}
  *_L
  \begin{vmatrix}
    \mathring{\mathbf{a}}_{22} & \mathring{\mathbf{a}}_{23} \\
    \mathring{\mathbf{a}}_{32} & \mathring{\mathbf{a}}_{33}
  \end{vmatrix}
  -\mathring{\mathbf{a}}_{12}
  *_L
  \begin{vmatrix}
    \mathring{\mathbf{a}}_{21} & \mathring{\mathbf{a}}_{23} \\
    \mathring{\mathbf{a}}_{31} & \mathring{\mathbf{a}}_{33}
  \end{vmatrix}
  +\mathring{\mathbf{a}}_{13}
  *_L
  \begin{vmatrix}
    \mathring{\mathbf{a}}_{21} & \mathring{\mathbf{a}}_{22} \\
    \mathring{\mathbf{a}}_{31} & \mathring{\mathbf{a}}_{32}
  \end{vmatrix}.
  \]
  By recursion, we can define the determinant of $\mathcal{A}\in \mathbb{C}^{n\times n\times n_3\times \cdots \times n_p}$ now,
  \[
  \mathtt{det}(\mathcal{A})=\sum_ {j=1}^{n}(-1)^{j-1}\mathring{\mathbf{a}}_{1j}*_L
  \mathtt{det}(\mathcal{M}_{1j}),
  \]
  where $\mathcal{M}_{1j}$ is a cofactor tensor.
\end{Definition}

We give an easy way to compute the determinant,
\begin{Theorem}\label{det}
  To compute the tensor determinant, we can first compute the determinant of $L(\mathcal{A})^{'(i)}$, denoted by $\mathring{\mathbf{d}}^{'(i)}$.
  Then compute $L^{-1}(\mathring{\mathbf{d}})$.
\end{Theorem}
\begin{proof}
  If $\mathcal{A}\in \mathbb{C}^{2\times 2\times n_3\times \cdots \times n_p}$,
  \begin{equation}\label{detproof1}
  \begin{split}
     \mathtt{det}(\mathcal{A})&=L^{-1}(L(\mathring{\mathbf{a}}_{11})\triangle_p L(\mathring{\mathbf{a}}_{22}))-L^ {-1}(L(\mathring{\mathbf{a}}_{21})\triangle_p L(\mathring{\mathbf{a}}_{12})) \\
       &=L^ {-1}(L(\mathring{\mathbf{a}}_{11})\triangle_p L(\mathring{\mathbf{a}}_{22})-L(\mathring{\mathbf{a}}_{21})\triangle_p L(\mathring{\mathbf{a}}_{12}))
  \end{split}
  \end{equation}
  Note that $L(\mathring{\mathbf{a}}_{11})^{'(i)} L(\mathring{\mathbf{a}}_{22})^{'(i)}-L(\mathring{\mathbf{a}}_{12})^{'(i)} L(\mathring{\mathbf{a}}_{21})^{'(i)}$ is the matrix determinant of $L(\mathcal{A})^{'(i)}$. If $n\geq 3$, the case of $\mathbb{C}^{(n-1)\times (n-1)\times n_3\times \cdots \times n_p}$ is right. For $\mathcal{A}\in \mathbb{C}^{n\times n\times n_3\times \cdots \times n_p}$, let the cofactor tensor of $\mathcal{A}$ be $\mathcal{M}_{1j}$ and the matrix determinant of $L(\mathcal{M}_{1j})^{'(i)}$ be $\mathring{\mathbf{d}}_{1j}^{'(i)}$. Then
  \begin{equation}\label{detproof2}
  \begin{split}
    L(\mathtt{det}(\mathcal{A}))&=\sum_ {j=1}^{n}(-1)^{j-1}L(\mathring{\mathbf{a}}_{1j})\triangle_p
  L(\mathtt{det}(\mathcal{M}_{1j}) \\
    &=\sum_ {j=1}^{n}(-1)^{j-1}L(\mathring{\mathbf{a}}_{1j})\triangle_p
  \mathring{\mathbf{d}}_{1j}
  \end{split}
  \end{equation}
  So we get
  \begin{equation}\label{dettens}
    L(\mathtt{det}(\mathcal{A}))^{'(i)}=\sum_ {j=1}^{n}(-1)^{j-1}L(\mathring{\mathbf{a}}_{1j})^{'(i)}
  \mathring{\mathbf{d}}_{1j}^{'(i)}=\mathring{\mathbf{d}}^{'(i)}
  \end{equation}
  Then $\mathtt{det}(\mathcal{A})=L^{-1}(\mathring{\mathbf{d}})$ and the conclusion follows.
\end{proof}

\begin{Theorem}
  $\mathtt{det}(\mathcal{A}*_L\mathcal{B})=\mathtt{det}(\mathcal{A})*_L
  \mathtt{det}(\mathcal{B})$, where $\mathcal{A}$, $\mathcal{B}\in\mathbb{C} ^{n\times n\times n_3\times \cdots \times n_p}$.
\end{Theorem}
\begin{proof}
  \begin{equation}\label{detproof3}
  \begin{split}
    \mathtt{det}(\mathcal{A})*_L
  \mathtt{det}(\mathcal{B})&= L^{-1}(L(L^{-1}(\mathring{\mathbf{d}}_1))\triangle_p L(L^{-1}(\mathring{\mathbf{d}}_2))) \\
    &=L^{-1}(\mathring{\mathbf{d}}_1\triangle_p \mathring{\mathbf{d}}_2)
  \end{split}
  \end{equation}
  where $\mathring{\mathbf{d}}_1$ and $\mathring{\mathbf{d}}_2$ are the tensors computed by the matrix determinant of $\mathcal{A}^{'(i)},\mathcal{B}^{'(i)}$ respectively. Let $\mathcal{C}=\mathcal{A}*_L\mathcal{B}$. Then $\mathcal{C}^{'(i)}=
  \mathcal{A}^{'(i)}\mathcal{B}^{'(i)}$. And we need to compute the determinant,
  $\mathtt{det}(\mathcal{C}^{'(i)})=
  \mathtt{det}(\mathcal{A}^{'(i)})\mathtt{det}(\mathcal{B}^{'(i)})$.
  We denote $\mathring{\mathbf{d}}_3$as the tensor computed by matrix determinant of $\mathcal{C}^{'(i)}$ so
  $\mathring{\mathbf{d}}_3=\mathring{\mathbf{d}}_1\triangle_p
  \mathring{\mathbf{d}}_2$. Then
  \begin{equation}\label{C=AB}
    \mathtt{det}(\mathcal{C})=L^ {-1}(\mathring{\mathbf{d}}_3)=\mathtt{det}(\mathcal{A})*_L
  \mathtt{det}(\mathcal{B}).
  \end{equation}
\end{proof}

\begin{Theorem}
\[
  \mathtt{det}(\mathcal{I})=L^{-1}(\mathbf{1}_{1\times 1\times n_3\times \cdots\times n_p}),
\]
where each element of $\mathbf{1}_{1\times 1\times n_3\times \cdots\times n_p}$ is 1.
\end{Theorem}
\begin{proof}
  From the definition of identity tensors, we know that $\mathtt{det}(L(\mathcal{I})^{'(i)})=1$. Then $\mathring{\mathbf{d}}=\mathbf{1}_{1\times 1\times n_3\times \cdots\times n_p}$. From the Theorem \ref{det},
  \begin{equation}\label{detproof4}
    \mathtt{det}(\mathcal{I})=L^ {-1}(\mathring{\mathbf{d}})=L^ {-1}(\mathbf{1}_{1\times 1\times n_3\times \cdots\times n_p}).
  \end{equation}
\end{proof}

\section{Tensor Nuclear Norm and Multi-rank}
In this section I will give the definition of tensor nuclear norm and tensor multi-rank. Then I prove a theorem to show their relationship.

The multi-rank can be generalized.
\begin{Definition}
  The multi-rank of $\mathcal{A}$ is defined as
  \[
  \mathbf{rank}_m(\mathcal{A})=(\mathrm{rank}(\mathcal{A}^{'(1)}),
  \mathrm{rank}(\mathcal{A}^{'(2)}),\cdots,
  \mathrm{rank}(\mathcal{A}^{'(n_3\cdots n_p)})).
  \]
\end{Definition}
Given a linear invertible transform, we can define a tensor nuclear norm and spectral norm,
\begin{Definition}
  The tensor nuclear norm and the tensor spectral norm under $*_L$ are respectively
  \[
  \|\mathcal{A}\|_{*,L}=\sum_{i=1}^{n_3\cdots n_p}\|L(\mathcal{A})^{'(i)}\|_{*}, \|\mathcal{A}\|=\|\mathtt{bdiag}({L(\mathcal{A})}^{'})\|.
  \]
\end{Definition}
\begin{figure}
  \centering
  \includegraphics[width=12cm]{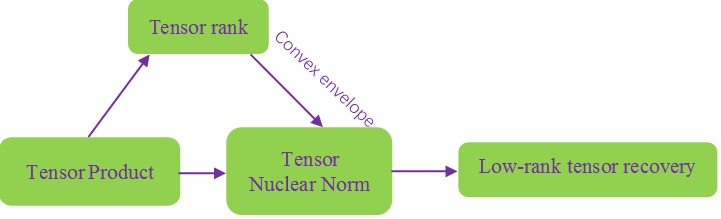}
  \caption{Overview of Low-Rank Tensor Recovery}\label{recovery}
\end{figure}
I briefly show the idea of low-rank tensor recovery in Figure \ref{recovery}.
We may use low-rank matrix recovery applying in unfolding tensors.
But it will lead to curse of dimensionality and damage the inherent structure.
We refer the readers to the recent work \cite{jiang2019robust,song2019robust} for details of low-rank tensor recovery models. Here I give a theorem to pave the way for the research about $p$-order tensor low-rank recovery.
\cite{song2019robust} gives a generalization under $*_u$.We need to use the fact that the tensor nuclear norm is the convex envelope of $\|\mathbf{rank}_m(\mathcal{A})\|_1$ to show the reasonableness of tensor nuclear norm including in the model.
Consider the biconjugate and prove a theorem which generalizes the work of \cite{song2019robust}.
\begin{Theorem}
  If $L$ is a unitary transform, $\|\mathcal{X}\|_{*,L}$ is the convex envelope of the $\ell_1$ norm of multi-rank on the set
  $\{\mathcal{X}|\|\mathcal{X}\|\leq 1\}$.
\end{Theorem}
\begin{proof}
  Let $\Upsilon(\mathcal{X})=\|\mathbf{rank} _m(\mathcal{X})\|_1$ and $n_{(2)}=\min(n_1,n_2)$. If
  $\|\mathcal{X}\|\leq 1$, the conjugate of $\Upsilon(\mathcal{X})$ on the set where spectral norm is 1 can be defined as,
  \begin{equation}\label{conjugatef}
    \Upsilon^{\#}(\mathcal{Y})=\sup_{\|\mathcal{X}\|\leq 1}(\mathrm{Re}(\langle \mathcal{Y},\mathcal{X}\rangle)-\|\mathbf{rank}_m(\mathcal{X})\|_1),
  \end{equation}
  Using von Neumann's trace inequality, we obtain
  \begin{equation}\label{vonneumann}
  \begin{split}
    \mathrm{Re}(\langle \mathcal{Y},\mathcal{X}\rangle) & =\mathrm{Re}(\langle L(\mathcal{Y}),L(\mathcal{X}\rangle))=
    \sum_{i=1}^{n_3\cdots n_p} \mathrm{Re}(\mathrm{tr}( (L(\mathcal{Y})^{'(i)})^H L(\mathcal{X}^{'(i)})))\\
     & \leq \sum_{i=1}^{n_3\cdots n_p}|\mathrm{tr}((L(\mathcal{Y})^{'(i)})^H L(\mathcal{X}^{'(i)}))|\\
     &\leq\sum_{i=1}^{n_3\cdots n_p}\sum_{j=1}^{n_{(2)}}\sigma_j(L(\mathcal{Y})^{'(i)})
     \sigma_j(L(\mathcal{X})^{'(i)})
  \end{split}
  \end{equation}
  where $\sigma_{j}(\cdot)$ denotes the $j$-th biggest singular value of the matrix. We perform the $p$-order t-SVD on $\mathcal{Y}$ and then $\mathcal{Y}=\mathcal{U}_y *_u
  \mathcal{S}_y *_u\mathcal{V}_y^H$.
  Let $\mathcal{U}_x=\mathcal{U}_y,
  \mathcal{V}_x=\mathcal{V}_y$ and choose $\mathcal{S}_x$ to make $\|\mathcal{X}\|\leq 1$. Then by the Theorem \ref{p,tSVD},
  \begin{equation}\label{supattained}
  \begin{split}
    \mathrm{Re}(\langle \mathcal{Y},\mathcal{X}\rangle) &
    =\sum_{i=1}^{n_3\cdots n_p} \mathrm{Re}(\mathrm{tr}( (L(\mathcal{Y})^{'(i)})^H L(\mathcal{X}^{'(i)}))) \\
     & =\sum_{i=1}^{n_3\cdots n_p} \mathrm{Re}(\mathrm{tr}( L(\mathcal{V}_y)^{'(i)}L(\mathcal{S}_y)^{'(i)}(L(\mathcal{U}_y)
     ^{'(i)})^H L(\mathcal{U}_x^{'(i)})L(\mathcal{S}_x^{'(i)})
     (L(\mathcal{V}_x^{'(i)}))^H)) \\
     & =\sum_{i=1}^{n_3\cdots n_p} \mathrm{Re}(\mathrm{tr}( L(\mathcal{V}_y)^{'(i)}L(\mathcal{S}_y)^{'(i)}L(\mathcal{S}_x^{'(i)})
     (L(\mathcal{V}_x^{'(i)}))^H))\\
     &=\sum_{i=1}^{n_3\cdots n_p} \mathrm{Re}(\mathrm{tr}( L(\mathcal{S}_y)^{'(i)}L(\mathcal{S}_x^{'(i)})))
     =\sum_{i=1}^{n_3\cdots n_p} \mathrm{tr}( L(\mathcal{S}_y)^{'(i)}L(\mathcal{S}_x)^{'(i)})\\
     &=\sum_{i=1}^{n_3\cdots n_p} \sum_{j=1}^{n_{(2)}}\sigma_j(L(\mathcal{Y})^{'(i)})
     \sigma_j(L(\mathcal{X})^{'(i)})
  \end{split}
  \end{equation}
  It shows we can obtain the equality. So
  \begin{equation}\label{upsilon}
  \Upsilon^{\#}(\mathcal{Y})=\sup_{\|\mathcal{X}\|\leq 1}(\sum_{i=1}^{n_3\cdots n_p} \sum_{j=1}^{n_{(2)}}\sigma_j(L(\mathcal{Y})^{'(i)})
     \sigma_j(L(\mathcal{X})^{'(i)})-\|\mathbf{rank}_m(\mathcal{X})\|_1),
  \end{equation}
  Because $\|\mathbf{rank}_m(\mathcal{X})\|_1$ is an integer, the range of it is $[0,n_ {(2)}n_3\cdots n_p]$.

  If $\|\mathbf{rank}_m(\mathcal{X})\|_1=0$, then all singular values should be 0, i.e., $\Upsilon^{\#}(\mathcal{Y})=0$.

  Let $n_{(2)} n_3\cdots n_p=N$. For all $1\leq r\leq N$ and an integer $r$, $\|\mathbf{rank}_m(\mathcal{X})\|_1=r$ and $\sigma_ {(i)}(\mathcal{A})$ denotes the $i$-th biggest singular values of all matrix slices of $\mathcal{A}$. So
  \begin{equation}\label{conjugatef1}
  \Upsilon^ {\#}(\mathcal{Y})=\max\{0,\sigma_{(1)}(L(\mathcal{Y}))-1,\cdots,
  \sum_{i=1}^r \sigma_{(i)}(L(\mathcal{Y}))-r\}
  \end{equation}
  Then when $\|\mathcal{Y}\|\leq 1$,$\Upsilon^ {\#}(\mathcal{Y})=0$. If there exists a positive integer $m$, not bigger than $r$, subject to,
  \begin{equation}\label{upsilon1}
  \sigma_{(m)}(L(\mathcal{Y}))>1,\sigma_ {(m)}(L(\mathcal{Y}))\leq 1,
  \end{equation}
  then
  $\Upsilon^ {\#}(\mathcal{Y})=\sum_{i=1}^m \sigma_{(i)}(L(\mathcal{Y}))-m$.
  Consider the conjugate of $\Upsilon^ {\#}(\mathcal{Y})$ further,
  \begin{equation}\label{biconjugate}
  \Upsilon^ {\#\#}(\mathcal{Z})=\sup_{\mathcal{Y}}(\mathrm{Re}(\langle \mathcal{Z},\mathcal{Y}\rangle)-\Upsilon^ {\#}(\mathcal{Y})),
  \end{equation}
  Perform the $p$-order t-SVD on $\mathcal{Z}$ yields that  $\mathcal{Z}=\mathcal{U}_z*_u \mathcal{S}_z *_u \mathcal{V}_z^H$. It is similar to get
  \begin{equation}\label{biconjugate1}
  \Upsilon^ {\#\#}(\mathcal{Z})=\sup_{\mathcal{Y}}(\sum_{i=1}^{n_3\cdots n_p} \sum_{j=1}^{n_{(2)}}\sigma_j(L(\mathcal{Y})^{'(i)})
     \sigma_j(L(\mathcal{Z})^{'(i)})-\Upsilon^ {\#}(\mathcal{Y}))
  \end{equation}
  Then if $\|\mathcal{Z}\|>1$, let the biggest singular value of all matrix slices of $L(\mathcal{Z})$ be on the $k$-th matrix slice. In other words,
  \begin{equation}\label{sigma}
  \sigma_ {(1)}(L(\mathcal{Z}))=\sigma_1(L(\mathcal{Z})^{'(k)}),
  \end{equation}
  Choose $\|\mathcal{Y}\|>1$ to make $\sigma_ {(1)}(L(\mathcal{Y}))=\sigma_1(L(\mathcal{Y})^{'(k)})$
  \begin{equation}\label{biconjugate2}
  \begin{split}
    \Upsilon^ {\#\#}(\mathcal{Z}) & = \sup_{\mathcal{Y}}(\sum_{i=1}^{n_3\cdots n_p} \sum_{j=1}^{n_{(2)}}\sigma_j(L(\mathcal{Y})^{'(i)})
     \sigma_j(L(\mathcal{Z})^{'(i)})-(\sum_{i=1}^{m}
     \sigma_{(i)}(L(\mathcal{Y}))-m))\\
    & = \sup_{\mathcal{Y}}\sigma_{(1)}(L(\mathcal{Y}))
    (\sigma_{(1)}(L(\mathcal{Z}))-1)+
    \sum_{i\neq k}\sum_{j=1}^{n_{(2)}} \sigma_j(L(\mathcal{Y})^{'(i)})\sigma_j(L(\mathcal{Z})^{'(i)})\\
    &+\sum_{j=2}^{n_{(2)}} \sigma_j(L(\mathcal{Y})^{'(k)})\sigma_j(L(\mathcal{Z})^{'(k)}) -(\sum_{i=2}^{m}
     \sigma_{(i)}(L(\mathcal{Y}))-m)
  \end{split}
  \end{equation}
$\sigma_{(1)}(L(\mathcal{Y}))$ can be sufficiently large, so the supremum is infinity. If $\|\mathcal{Z}\|\leq 1$ and $\|\mathcal{Y}\| \leq 1$, we get $\Upsilon^ {\#}(\mathcal{Y})=0$. The supremum is obtained if $\sigma_ {(i)}(L(\mathcal{Y}))=1,i=1,\cdots,N$ so
\begin{equation}\label{biconjugate3}
\Upsilon^ {\#\#}(\mathcal{Z})= \sum_{i=1}^{N} \sigma_i(L(\mathcal{Z}))=\|Z\|_{*,u}
\end{equation}
If $\|\mathcal{Y}\| > 1$, let the singular value $\sigma_ {(i)}(L(\mathcal{Y}))$ of matrix slice of $L(\mathcal{Y})$ correspond to $\sigma_{k(i)}(L(\mathcal{Z}))$ with regard to $\sum_{i=1}^{n_3\cdots n_p}\sum_{j=1}^{n_{(2)}}\sigma_j(L(\mathcal{Y})^{'(i)})
\sigma_j(L(\mathcal{Z})^{'(i)})$, and then
\[
\sum_ {i=1}^{m}(\sigma_{k(i)}(L(\mathcal{Z}))-1)(\sigma_{(i)}(L(\mathcal{Y}))-1)\leq 0,\mbox{and}
\sum_ {i=m+1}^{N}\sigma_{k(i)}(L(\mathcal{Z}))(\sigma_{(i)}
(L(\mathcal{Y}))-1)\leq 0,
\]
so
\begin{equation}\label{endproof}
  \begin{split}
   \sum_{i=1}^{n_3\cdots n_p} & \sum_{j=1}^{n_{(2)}}\sigma_j(L(\mathcal{Y})^{'(i)})
     \sigma_j(L(\mathcal{Z})^{'(i)})-\sum_{i=1}^{m}
     (\sigma_{(i)}(L(\mathcal{Y}))-1)\\
  & = \sum_{i=1}^{m}
     \sigma_{k(i)}(L(\mathcal{Z}))\sigma_{(i)}(L(\mathcal{Y}))+
     \sum_{i=m+1}^{N}\sigma_{k(i)}(L(\mathcal{Z}))\sigma_{(i)}(L(\mathcal{Y}))\\
     &-\sum_{i=1}^{m}
     (\sigma_i(L(\mathcal{Y}))-1)
     -\sum_{i=1}^{N}
     \sigma_{(i)}(L(\mathcal{Z}))+\sum_{i=1}^{N}
     \sigma_{(i)}(L(\mathcal{Z}))\\
  & =\sum_{i=1}^{N}\sigma_{(i)}(L(\mathcal{Z}))+\sum_{i=1}^{m}
     (\sigma_{k(i)}(L(\mathcal{Z}))-1)(\sigma_{(i)}(L(\mathcal{Y}))-1)\\
     &+\sum_{i=m+1}^{N}
     \sigma_{k(i)}(L(\mathcal{Z}))(\sigma_{(i)}(L(\mathcal{Y}))-1)
     \leq \sum_{i=1}^{N}\sigma_{(i)}(L(\mathcal{Z}))
\end{split}
\end{equation}
On the set $\{\|\mathcal{Z}\|\leq 1\}$,$\Upsilon^ {\#\#}(\mathcal{Z})=\sum_{i=1}^{N}\sigma_{(i)}(L(\mathcal{Z}))=
\|\mathcal{Z}\|_{*,u}$. We get the conclusion exactly.
\end{proof}

\section{Conclusion}
In this paper, I propose a generalized kind of $p$-order tensor product using invertible linear transform. Given one transform, we may compute the tensor product efficiently. Also, after defining the tensor nuclear norm under this product and ranks, I prove a theorem to show the relationship between the multi-rank and tensor nuclear norm, which will facilitate the research of $p$-order tensor low-rank recovery. I also leave this direction to the future work. The generalized tensor determinant is defined as a tube. Its meaning needs to be thoroughly investigated in the future. Moreover, it is a little tricky to choose a good transform for a given data set. Given the insight from the theorem in \cite{kernfeld2015tensor}, this topic should be studied deeper. In an operator perspective like \cite{kilmer2013third}, $p$-order tensor products may be analyzed further.

\section*{Acknowledgement}
Thanks for the comment from Prof. Daniel Kuhn, Prof. Michael Ng, Prof. Chunlin Wu and Doc. Dirk Lauinger.

\bibliographystyle{siam}
\bibliography{./reference.bib}

\begin{thebibliography}{10}

\bibitem{belkin2003laplacian}
{\sc M.~Belkin and P.~Niyogi}, {\em Laplacian eigenmaps for dimensionality
  reduction and data representation}, Neural computation, 15 (2003),
  pp.~1373--1396.

\bibitem{chang2017weighted}
{\sc Y.~Chang, L.~Yan, H.~Fang, S.~Zhong, and Z.~Zhang}, {\em Weighted low-rank
  tensor recovery for hyperspectral image restoration}, arXiv preprint
  arXiv:1709.00192,  (2017).

\bibitem{eckart1936approximation}
{\sc C.~Eckart and G.~Young}, {\em The approximation of one matrix by another
  of lower rank}, Psychometrika, 1 (1936), pp.~211--218.

\bibitem{ely20155d}
{\sc G.~Ely, S.~Aeron, N.~Hao, and M.~E. Kilmer}, {\em 5d seismic data
  completion and denoising using a novel class of tensor decompositions},
  Geophysics, 80 (2015), pp.~V83--V95.

\bibitem{haastad1989tensor}
{\sc J.~H{\aa}stad}, {\em Tensor rank is np-complete}, in International
  Colloquium on Automata, Languages, and Programming, Springer, 1989,
  pp.~451--460.

\bibitem{hitchcock1927expression}
{\sc F.~L. Hitchcock}, {\em The expression of a tensor or a polyadic as a sum
  of products}, Journal of Mathematics and Physics, 6 (1927), pp.~164--189.

\bibitem{jiang2019robust}
{\sc Q.~Jiang and M.~Ng}, {\em Robust low-tubal-rank tensor completion via
  convex optimization}, in Proceedings of the 28th International Joint
  Conference on Artificial Intelligence, Macao, China, 2019, pp.~10--16.

\bibitem{kernfeld2015tensor}
{\sc E.~Kernfeld, M.~Kilmer, and S.~Aeron}, {\em Tensor--tensor products with
  invertible linear transforms}, Linear Algebra and its Applications, 485
  (2015), pp.~545--570.

\bibitem{kiers2000towards}
{\sc H.~A. Kiers}, {\em Towards a standardized notation and terminology in
  multiway analysis}, Journal of Chemometrics: A Journal of the Chemometrics
  Society, 14 (2000), pp.~105--122.

\bibitem{kilmer2013third}
{\sc M.~E. Kilmer, K.~Braman, N.~Hao, and R.~C. Hoover}, {\em Third-order
  tensors as operators on matrices: A theoretical and computational framework
  with applications in imaging}, SIAM Journal on Matrix Analysis and
  Applications, 34 (2013), pp.~148--172.

\bibitem{kilmer2011factorization}
{\sc M.~E. Kilmer and C.~D. Martin}, {\em Factorization strategies for
  third-order tensors}, Linear Algebra and its Applications, 435 (2011),
  pp.~641--658.

\bibitem{kolda2009tensor}
{\sc T.~G. Kolda and B.~W. Bader}, {\em Tensor decompositions and
  applications}, SIAM review, 51 (2009), pp.~455--500.

\bibitem{kruskal1989rank}
{\sc J.~B. Kruskal}, {\em Rank, decomposition, and uniqueness for 3-way and
  n-way arrays}, Multiway data analysis,  (1989), pp.~7--18.

\bibitem{li2013compressed}
{\sc X.~Li}, {\em Compressed sensing and matrix completion with constant
  proportion of corruptions}, Constructive Approximation, 37 (2013),
  pp.~73--99.

\bibitem{lu2019tensor}
{\sc C.~Lu, J.~Feng, W.~Liu, Z.~Lin, S.~Yan, et~al.}, {\em Tensor robust
  principal component analysis with a new tensor nuclear norm}, IEEE
  transactions on pattern analysis and machine intelligence,  (2019).

\bibitem{martin2013order}
{\sc C.~D. Martin, R.~Shafer, and B.~LaRue}, {\em An order-p tensor
  factorization with applications in imaging}, SIAM Journal on Scientific
  Computing, 35 (2013), pp.~A474--A490.

\bibitem{semerci2014tensor}
{\sc O.~Semerci, N.~Hao, M.~E. Kilmer, and E.~L. Miller}, {\em Tensor-based
  formulation and nuclear norm regularization for multienergy computed
  tomography}, IEEE Transactions on Image Processing, 23 (2014),
  pp.~1678--1693.

\bibitem{song2019robust}
{\sc G.~Song, M.~K. Ng, and X.~Zhang}, {\em Robust tensor completion using
  transformed tensor svd}, arXiv preprint arXiv:1907.01113,  (2019).

\bibitem{tucker1966some}
{\sc L.~R. Tucker}, {\em Some mathematical notes on three-mode factor
  analysis}, Psychometrika, 31 (1966), pp.~279--311.

\bibitem{zhang2016exact}
{\sc Z.~Zhang and S.~Aeron}, {\em Exact tensor completion using t-svd}, IEEE
  Transactions on Signal Processing, 65 (2016), pp.~1511--1526.

\bibitem{zhang2014novel}
{\sc Z.~Zhang, G.~Ely, S.~Aeron, N.~Hao, and M.~Kilmer}, {\em Novel methods for
  multilinear data completion and de-noising based on tensor-svd}, in
  Proceedings of the IEEE conference on computer vision and pattern
  recognition, 2014, pp.~3842--3849.

\end{thebibliography}

\end{document}